\documentclass[12pt, ]{amsart}

\usepackage{fullpage}

\usepackage{amsmath,amsfonts,amssymb,graphics,amsthm, enumerate}

\usepackage{mdframed} 

\usepackage[normalem]{ulem} 
\usepackage{hyperref}
\hypersetup{
    colorlinks=true,
    linkcolor=blue,
    citecolor=red,
    urlcolor=blue,
    pdfborder={0 0 0}
}            

\usepackage{graphicx}
\usepackage{xcolor}

\usepackage[font=sf, labelfont={sf,bf}, margin=1cm]{caption}

\usepackage{cleveref}
  \crefname{theorem}{Theorem}{Theorems}
  \crefname{thm}{Theorem}{Theorems}
  \crefname{lemma}{Lemma}{Lemmas}
  \crefname{lem}{Lemma}{Lemmas}
  \crefname{remark}{Remark}{Remarks}
  \crefname{prop}{Proposition}{Propositions}
  
  \crefname{defn}{Definition}{Definitions}
  \crefname{corollary}{Corollary}{Corollaries}
  \crefname{section}{Section}{Sections}
  \crefname{figure}{Figure}{Figures}

\newtheorem{thm}{Theorem}[section]
\newtheorem{lemma}[thm]{Lemma}
\newtheorem{corollary}[thm]{Corollary}

\numberwithin{equation}{section}

\theoremstyle{definition}
\newtheorem{remark}[thm]{Remark}

\newtheorem{prop}[thm]{Proposition}

\def\cD{\mathcal{D}}
\def\cC{\mathcal{C}}

\def\w{\omega}

\def\C{\mathbb{C}}
\def\D{\mathbb{D}}

\def\H{\mathbb{H}}

\def\R{\mathbb{R}}

\def  \p- {p\textunderscore}

\def\th1{{\theta_1}}

\DeclareMathOperator{\re}{Re}

\DeclareMathOperator{\diam}{diam}

\DeclareMathOperator{\im}{Im}

\usepackage{tikz}

\newcommand{\lbda}{\lambda}

\title{Loewner Equation driven by complex-valued functions}

\author{ Huy Tran}
\address{UCLA Department of Mathematics, Los Angeles, CA 90095-1555. }
\email{tvhuy@math.ucla.edu}

\date{\today}

\begin{document}
\maketitle
\begin{abstract} Consider the Loewner equation associated to the upper-half plane. Normally this equation is driven by a real-valued function. In this paper, we show that when the driving function is complex-valued with small $1/2$-H{\"o}lder norm, the corresponding hull is a quasi-arc, hence is a simple curve. We also study how the hull changes with respect to complex parameters and make a connection between Loewner equation and complex dynamics.
\end{abstract}
\tableofcontents

%
\section{Introduction}
Consider the Loewner equation version for upper half plane $\mathbb{H}$, that is, consider a continuous function $\lambda:[0,1]\to \R$ and the following ordinary differential equation
\begin{eqnarray}
\partial_t g_t(z) & = & \frac{2}{g_t(z)-\lambda(t)}\label{e:LE1}, \\
g_0(z) & = & z, \nonumber
\end{eqnarray}
for each $z\in \mathbb{H}$.

For each $z\in \mathbb{H}$, there exists a unique $T_z\in (0,1]$ such that the equation (\ref{e:LE1}) has the unique solution up to $T_z$ but not further.  For each $t\geq 0$, define the hull $K_t=\{z\in \mathbb{H}:T_z\leq t\}$. It can be proved that  $g_t$ is a conformal map from $\H\backslash K_t$ onto $\H$; see \cite{L}.

If there is a curve $\gamma:[0,1]\to \overline{\H}$ such that $\H\backslash K_t$ is the unbounded component of $\H\backslash \gamma([0,t])$ for $t\in [0,1]$, then we say that $\gamma$ is a Loewner curve generated by $\lbda$, or that $\lambda$ generates $\gamma$. In cases to emphasize the dependency on $\lbda$, we use the notation $\gamma^{\lbda}$. This curve can be computed from $g_t$:
\begin{eqnarray}
\label{e:curve}
\gamma(t) &=&\lim_{y\to 0^+} g_t^{-1}(\lambda(t)+iy)\mbox{ for } t\in [0,1].
\end{eqnarray}

The existence of the curve $\gamma$ holds when $\lambda$ has 1/2-H{\"o}lder norm less than 4 (See \cite{RS}, \cite{Lind}, and \cite{RTZ}),
$$||\lbda||_{1/2}:=\sup_{s,t\in [0,1]:s\neq t}\frac{|\lbda(t)-\lbda(s)|}{|t-s|^{1/2}}<4,$$
or when $\lbda$ is a multiplication of Brownian motion (See \cite{RS}). In the former case, the curve $\gamma$ is a quasi-arc, which means there exists a constant $C>0$ such that for any $x$ and $y$ in $ \gamma$, the diameter of the path in $\gamma$ connecting $x$ to $y$ is  less than that of $C|x-y|$. This is where one of the main motivations for the project started. 

The quasi-arcs and quasi-circles are  objects in complex dynamics associated with a notion called holomorphic motion. This was first developed by Ma\~n\'e, Sad and Sullivan \cite{MSS} to understand the Julia set of polynomials when one changes its coefficients in a holomorphic way. Another way to explain is that holomorphic motion describes the analytic movements of a subset in the complex plane. More precisely,  let $A$ be a subset of $\C$. A holomorphic motion of $A$ is a map $f:\D\times A \to \mathbb{C}$ such that
\begin{enumerate}[(i)]
\item for any fixed $a\in A$, the map $\alpha\mapsto f(\alpha,a)$ is holomorphic in $\D$.
\item for any fixed $\alpha\in \D$, the map $a\mapsto f(\alpha,a)=f_\alpha(a)$ is an injection and
\item the mapping $f_0$ is the identity on $A$.
\end{enumerate}

Just from very few assumptions, Ma\~n\'e, Sad, and Sullivan \cite{MSS}, and then  Slodkowski \cite{Slod} showed that 

\begin{thm}\label{thm:MSS} If $f:\D\times A \to \C$ is a holomorphic motion, then $f$ has an extension to $F:\D\times \C\to \C$  such that 
\begin{enumerate}[(i)]
\item $F$ is a holomorphic motion of $\C$,
\item each $F_\alpha(\cdot):\C\to \C$ is quasi-symmetric,
\item $F$ is jointly continuous in $(\alpha,a)$.
\end{enumerate}

\end{thm}

In particular, when $A$ is a line segment (or a circle), then $f_\alpha(A)$ is a quasi-arc (or quasi-circle respectively). This makes the holomorphic motion become an important tool in quasi-conformal mapping theory and in complex dynamics. As an example, in \cite{MSS}, it was shown that the Julia sets of polynomials $z^2+c$ are all quasi-circles when $c$ is in the main cardioid of the Mandelbrot set. One can also show that the Koch snowflake is a quasi-arc by putting it into a holomorphic motion. Conversely, if a curve is  a quasi-circle, by using the Beltrami equation, one can see that there is a holomorphic motion in which the curve is an image of the unit circle.

Thus from the complex dynamic point of view, one can ask: Is there a way to put the quasi-arc curve $\gamma$ generated from the Loewner equation into a holomorphic motion of a line segment such that when the complex parameter $\alpha$ changes in $\D$, the motion  is tied up with the Loewner equation in a certain way? 

Here we note that the authors in \cite{RS} used techniques from the quasi-conformal mapping theory but not holomorphic motion. The question is related to studying (\ref{e:LE1}) with complex-valued $\lambda$. In this context, one has to define $K_t$ differently since the upper half-plane plays no special role. In particular, suppose $\lambda$ is complex-valued and continuous. It is still true from standard ODE theory that for each $z\in \C$, there exists a unique $T_z\in [0,1]$ such that the equation (\ref{e:LE1}) has the unique solution up to $T_z$ but not further.  For each $t\geq 0$, 
define
\begin{eqnarray}
L_t & = & \{z\in \C: T_z\leq t\}.
\end{eqnarray}
\noindent
One can ask: Is $L_t$ a curve when $||\lambda||_{1/2}$ is small? 

Note that if $\lambda$ is real-valued, then 
\begin{eqnarray}\label{e:LK_t}
L_t&=&\overline{K_t}\cup \overline{K_t^*}
\end{eqnarray}
 where $K_t^*$ is the reflection of $K_t$ about the real line.  This question is actually not new. It was studied by O. Schramm and S. Rohde \cite{RS13}. 

The main purpose of this paper is to answer the two questions above. Explicitly, we show that

\begin{thm}\label{t:L_t curve}
Suppose that $\lbda$ is a complex-valued function with  1/2-H{\"o}lder norm less than $\sigma$. If $\sigma$ is small, then there exists a quasi-arc $\gamma:[-1,1]\to \C$ such that $L_t = \gamma[-t,t]$ for every $t\in [0,1]$.
\end{thm}

Also we show that

\begin{thm}\label{t:motion} Let $\lbda$ be a complex-valued driving function with  1/2-H{\"o}lder norm less than $\sigma$. If $\sigma$ is small, then there exists a map $\gamma:\D\times [0,1]\to \C$ such that $\gamma^{(\alpha)}(t):=\gamma(\alpha,t)$ satisfies:
\begin{enumerate}[(i)]
\item for any fixed $t\in [0,1]$, the map $\alpha\mapsto \gamma^{(\alpha)}(t)$ is holomorphic in $\D$,
\item for any fixed $\alpha\in \D$, the map $t\mapsto \gamma^{(\alpha)}(t)$ is injective,
\item  $\gamma^{(0)}(t)=2i\sqrt{t}$.
\item when $\alpha\in \D\cap \mathbb{R}$, $\gamma^{(\alpha)}$ is generated by $\alpha \lambda$ from the Loewner equation,
\end{enumerate}
\end{thm}

In particular, if we define $F:\D\times [0,2i]\to \C$ such that 
$$F(\alpha,a)=\gamma^{(\alpha)}(\frac{-a^2}{4}),$$
then $F$ is a holomorphic motion of $[0,2i]$. Hence Theorem \ref{thm:MSS} immediately implies that 

\begin{corollary} \label{c:quasi}
For all $\alpha\in \D$, the curve $\gamma^{(\alpha)}$ is a quasi-arc. In particular, $\gamma$ is a quasi-arc.
\end{corollary}

\begin{remark} The continuity in $t$ of $\gamma^{(\alpha)}(t)$ is a part of the corollary.
\end{remark}
\begin{remark} For concreteness, the constant $\sigma$ in Theorems  \ref{t:L_t curve} and \ref{t:motion} can be set $\sigma=1/3$ which is non-optimal from the proof. We do not know if one can take $\sigma=4$ which is the optimal $1/2$-H{\"o}lder norm in Marshall and Rohde's theorem (\cite{Lind}).
\end{remark}
\begin{remark}
There is nothing special about the unit disk $\D$ in Theorem \ref{t:motion}. In fact, the proof shows that we can take $\alpha$ to be in a bigger set which contains $\D$.
\end{remark}

\begin{remark} A simple example for the theorem is that when $\lbda(t)=c$ for all $t\in [0,1]$,
$$\gamma^{(\alpha)}(t)=2i\sqrt{t}+\alpha c\mbox{ for } t\in [0,1], \alpha\in \D.$$
See (\ref{e:ODEf2}) and (\ref{eqn:gamma}).

\end{remark}


Let us explain the main ideas in the paper. We will prove that the limit as in (\ref{e:curve}) is well-defined for each $t\in [0,1]$. To do this we study the backward Loewner equation. For each $t\in [0,1]$, consider 
\begin{eqnarray}
\partial_u h_{u,t}(z) &=& \frac{-2}{h_{u,t}(z)-\lbda(t-u)}, \label{e:ODEhu}\\
h_{0,t}(z) & = & z \in \C. \nonumber
\end{eqnarray}
There is a sign difference between (\ref{e:ODEhu}) and (\ref{e:LE1}). 

In the classical setting, i.e. $\lambda$ is real-valued, the equation has a unique solution $h_{u,t}(z), u\in [0,t]$ for any $z\in \H$. Furthermore, one can show that $h_{u,t}(z)$ is conformal with respect to $z\in \H$ and
$$h_{t,t}=g_t^{-1}.$$
In particular,  $\gamma(t)=\lim_{y\to 0^+} h_{t,t}(\lbda(t)+iy)$. We can renormalize $h$ by defining $f_{u,t}(z)=h_{u,t}(z+\lbda(t))-\lbda(t-u)$. Then
\begin{eqnarray}
\partial_u (f_{u,t}(z)+\lbda(t-u)) &=&\frac{-2}{f_{u,t}(z)}\label{e:LE3},\\
f_{0,t}(z)&=&z. \nonumber
\end{eqnarray}

When $\lambda$ is complex-valued, it is not clear that (\ref{e:LE3}) has solution for any $z=\lambda(t)+iy$. We will show that it is still true provided $||\lambda||_{1/2}$ is small.

There are two key ideas. The first is that the equation (\ref{e:LE3}) is understood pretty well quantitatively in the  real-valued case by the paper \cite{RTZ}. The other idea is that  we can compare the complex-valued case to the real-valued one by a certain Gronwall-type lemma; see Lemma \ref{l:Gronwall}.

The organization of the paper is as follows. In Section \ref{s:proof}, we show the existence of the curve $\gamma$ in the sense of (\ref{e:curve}). Then we prove Theorem \ref{t:motion}. In Section \ref{s:hull}, we investigate the hull $L_t$ and prove Theorem \ref{t:L_t curve}.

{\bf Acknowledgments.} H.T. is partially supported by NSF grant DMS-1162471. The author is indebted to S. Rohde for 	numerous discussions and the suggestion to study the complex Loewner evolution.  He also thanks M. Bonk for his encouragement while the project was being done.
\section{A lemma} \label{sec:pre}

The following comparison lemma, which will be used in Lemma \ref{lm:existence}, is inspired by the proof of Theorem 3.4 in \cite{Wong}. This is also one of the main tools in \cite{LT}. 

\begin{lemma} \label{l:Gronwall}
Let $Z:[0,u_0]\to \C$ be a solution to
$$Z'(u)=P(u)Z(u) - P(u)Q(u),$$
with $|P|\leq -C\re P$ and $|Q(v)-Q(0)|\leq \w(v) $ on $[0,u_0]$, where $\w$ is a non-decreasing function.   Then
$$|Z(u) - Q(u)|\leq |Z(0)-Q(0)|+(C+1) \w(u) \mbox{ for all } u\in [0,u_0] .$$
\end{lemma}
\begin{proof}
Let $\mu(u)=\int^u_0 -P(v) \, dv$. We have
\begin{align*}
Z(u) &=e^{-\mu(u)}Z(0) + e^{-\mu(u)}\int^u_0 e^{\mu(v)}(-P Q ) \, dv \\
 &=e^{-\mu(u)}(Z(0)-Q(0))+Q(0)+e^{-\mu(u)}\int^u_0 e^{\mu(v)}(-P)[Q-Q(0)] \, dv.
\end{align*}
Therefore,
$$|Z(u)-Q(u)|\leq e^{-\re\mu(u)}|Z(0)-Q(0)| + |Q(0)-Q(u)| + e^{-\re\mu(u)}\int^u_0 e^{\re \mu(v)}C (-\re P)\w(u) \, dv $$
$$\leq |Z(0)-Q(0)|+ (C+1)\w(u).$$

\end{proof}


\section{Backward Loewner equation driven by complex-valued functions} \label{s:proof}
Fix small $\sigma$. Let $\Lambda_\sigma$ be the set of complex-valued functions defined on $[0,1]$ such that its 1/2-H{\"o}lder norm is less than $\sigma$
$$\{\lbda:[0,1]\to \C: ||\lbda||_{1/2}<\sigma\}.$$
 For each $\theta > 0$, define $\cC_\theta=\{z\in \H: |\re z|< \theta \im z\}$ be an upside-down cone based at $z=0$ with the ``angle'' $\theta$. The set $\cC_\theta$ is open in $\C$ and does not contain zero. For $\theta=0$, define $\cC_0=i\R^+$. Fix $\lbda\in \Lambda_\sigma$. In this section, we do the following steps.
\begin{enumerate}[(i)]
\item Show that when $\theta_1$ is positive and small enough, then for   $z\in \cC_\th1$, $0\leq t\leq 1$,  the equation
\begin{eqnarray}
\partial_u (f(u,t, z)+ \lbda(s-u)) &= &\frac{-2}{f(u,t,  z)}, \label{e:ODEf2}\\
f(0,t, z) &= &z \nonumber
\end{eqnarray}
has the unique solution $f(u,t,z)$, $u\in [0,t]$.

\item Show that  $$\lim_{y\to  0^+ } f(u,t, iy)$$
exists. And the limit is denoted by $f(u,t,0^+)$. 
Then define 
\begin{equation}\label{eqn:gamma}
\gamma^{(\lbda)}(t)=f(t,t,0^+)+ \lbda(0)
\end{equation} for each $t\in [0,1]$. 
\item  Show that  the map $t\mapsto \gamma^{(\lbda)}(t)$ is injective.

\item Show that when $\lbda$ is a real-valued function, then  the curve $(\gamma^{(\lbda)}(t))_{t\in [0,1]}$ is the same as $\tilde\gamma$ which is the curve generated by $\lbda$ in the Loewner equation (\ref{e:LE1}).
\end{enumerate}

These steps are respectively proved  in Sections \ref{sec:existence}-\ref{s:lbda_real}. Then we prove Theorem \ref{t:motion} in Section \ref{s:motion}. We remark that the continuity of $\gamma^{(\lbda)}$  follows from Corollary \ref{c:quasi}.


\subsection{Existence of the solution to  the initial ODE} \label{sec:existence}
In this section, fix $t\in [0,1]$, and fix $z\in \cC_{\th1}$ where $\cC_{\th1}$ is the upside down cone based at 0
 $$\cC_{\th1}=\{w\in \H: |\re w|<\th1 \im w\}.$$ 

We consider the following ODE
\begin{eqnarray}
\frac{d}{du} (A + \lbda(t-u)) &=& \frac{-2}{A} \label{eqn:ODEA},\\
A(0)&=&z.\nonumber
\end{eqnarray}

%
%
%

\begin{lemma}\label{lm:existence}
Suppose that $\th1$ is small enough depending only on $\sigma$. Then this equation has a unique solution $A(u), 0\leq u\leq t$, for given $t, z$. Moreover, there exist $\tau(\th1)$ and $\nu(\theta_1)$ depending only on $\theta_1$ (and $\sigma$) such that
\begin{eqnarray*}
A(u) &\in &\cC_{\tau(\th1)} \label{eqn:coneA},\\
\im A(u) &\geq &\nu(\th1)\sqrt{u}. \nonumber
\end{eqnarray*}
\end{lemma}
\begin{proof}
This ODE has solution on small time $u$. By the ODE theory, the solution continue to exist uniquely as long as $A(u)  \neq 0.$ We will show a stronger statement: the solution flow always stays in a fixed cone whose angle is depending on $\theta_1$ and $\sigma$. 

The idea of the lemma is as follows. We compare the solution $A$ to $B(u)=i\sqrt{\hat y_0^2+4u}$ which is the solution of the ODE
\begin{eqnarray*}
\frac{dB}{du}  &=& \frac{-2}{B}, \\
B(0)&=&i\hat y_0, \nonumber
\end{eqnarray*}
for a well-chosen $\hat y_0>0$ depending only on $z$. Since $B(u)$ is flowing up as $u$ increases, $A(u)$ is dragged along and  stays in  a bigger but fixed cone $\cC_{\theta_2}$. 

Now by a topological argument, it suffices to show that for any $t_1\in [0,t)$ if
\begin{equation}\label{eqn:coneA}
A(u)\in \cC_{\theta_2}\mbox{ for all } u \in [0,t_1),
\end{equation}
with $\theta_2$ chosen later, then $A(t_1)\in \cC_{\theta_2}$. Suppose that (\ref{eqn:coneA}) holds.

Let $P(u) = \frac{2}{A(u)B(u)}$. Then
$$|P(u)|\leq -C_2\re P(u) \mbox{ for all } u\in [0,t_1),$$
where 
$$C_2=\sqrt{1+\theta_2^2}.$$ 
 Since
$$\partial_u (A(u)-B(u) +\lbda(s-u))=P(u) (A(u)-B(u)),$$
by Lemma \ref{l:Gronwall}, we derive
\begin{equation}\label{eqn:A-B}
|A(u)-B(u)|\leq |A(0)-B(0)| + (C_2+1) \sigma \sqrt{u}.
\end{equation}
Hence for all $u\in [0,t_1)$,
\begin{eqnarray*}
\frac{|\re A(u)|}{\im A(u)} & \leq & \frac{ |A(0)-B(0)|+(C_2+1)\sigma\sqrt{u} }{\sqrt{\hat{y}_0^2+4u}- |A(0)-B(0)|-(C_2+1)\sigma\sqrt{u}}.
\end{eqnarray*}
We note that the right-hand side is strictly less than $\theta_2$ for all $u\geq 0$. Indeed, it is equivalent to
\begin{equation*}
(1+\theta_2)|A(0)-B(0)| + (1+\theta_2)(\sqrt{1+\theta_2^2}+1)\sigma\sqrt{u}< \theta_2\sqrt{4u+\hat y_0^2}.
\end{equation*}

Since $a+b\sqrt{u}< c\sqrt{u+d}$ for all $u\geq 0$ if $a< \sqrt{c^2-b^2}\sqrt{d}$ and $a,b,c,d\geq 0$, the above inequality holds if

$$\frac{|A(0)-B(0)|}{\hat y_0} \frac{1+\theta_2}{\theta_2}< \sqrt{1-\frac{\sigma^2(1+\theta_2)^2}{4\theta_2^2} (\sqrt{1+\theta_2^2}+1)^2},$$
and the expression under the root sign on the right-hand side is positive.

By a geometric argument,
$$\min_{\hat y_0>0} \frac{|A(0)-i\hat y_0|}{\hat y_0}=\sin (\tan^{-1}\frac{|x|}{y})< \frac{\th1}{\sqrt{1+\theta_1^2}},$$
where $A(0)=z=x+iy\in \cC_\th1$ and the minimum happens when $\hat y_0=\frac{|A(0)|}{\cos (\tan^{-1}(|x|/y))}$. This leads to the constrain 
\begin{equation}\label{eqn:constrain}
 \frac{\th1}{\sqrt{1+\theta_1^2}}\frac{1+\theta_2}{\theta_2}< \sqrt{1-\frac{\sigma^2(1+\theta_2)^2}{4\theta_2^2} (\sqrt{1+\theta_2^2}+1)^2}.
 \end{equation}
For a given small $\sigma$, one can choose small $\theta_1$ depending on $\sigma$, and then choose $\tau(\th1):=\theta_2$ depending on $\theta_1$ and $\sigma$ such that the above constrain is true.

It follows from the proof that
\begin{eqnarray*}
\im A(u) & \geq & \sqrt{\hat{y}_0^2+4u}- |A(0)-B(0)|-(C_2+1)\sigma\sqrt{u} \\
&> &\frac{1}{\theta_2} (|A(0)-B(0)|+(C_2+1)\sigma\sqrt{u})\\
&\geq & \frac{1}{\theta_2}(\sqrt{1+\theta_2^2}+1)\sigma\sqrt{u}=:\nu(\th1)\sqrt{u}~~~\mbox{ for all } u\in [0,t_1).
\end{eqnarray*}
Hence $A(t_1)$ is well-defined and in $\cC_{\theta_2}$. This concludes the lemma.
\end{proof}

\begin{remark} \label{rm:quant}
 \begin{enumerate}[(i)]
\item The smaller $\sigma$ is, the bigger possible range of $\theta_1$ is. More quantitatively, for given small $\sigma$, we can choose $\theta_1(\sigma)$  such that (\ref{eqn:constrain})  holds for some $\theta_2>0$ and such that  $\lim_{\sigma\to 0^+} \theta_1(\sigma)\to \infty.$ This fact is needed later in Section \ref{s:hull}.

\item When $A(0)\in \cC_0$, then (\ref{eqn:constrain}) is always true for $\theta_2=1$ and $\sigma$ small, say $\sigma\leq 1/3$. Hence we assume $\tau(0)\leq 1$. Also, when $\sigma\to 0^+$, we can choose $\tau(0)$ such that $\tau(0)\to 0$. Thus, we assume $\tau(0)\leq \theta_1(\sigma)$.

\end{enumerate}
\end{remark}
\begin{remark} \label{r:cone1}
The lemma shows the existence of the solution $f(u,t, z) :=A(u), 0\leq u\leq t,$ when $z\in \cC_{\th1}$.  We note that $f$ also depends on the driving function $\lbda$. But we omit this notation since it is clear from the context. 
Some properties of $f(u,t,z)$:\begin{enumerate}[(i)]

\item $f(u,t, z)\in \cC_{\tau(\th1)}$. If $z\in i\R^+$, then $f(u,t,z)\in \cC_{\tau(0)}\subset \cC_1.$

\item $\im f(u,t, z)> \nu(\th1)\sqrt{u}$.

\item By the dependency of solutions of ODE on parameters (\cite[Chapter 1, Theorem 8.4]{CL}), the map $z\mapsto f(u,t,z)$   is analytic in $\cC_\th1$. Also the map $u\mapsto \partial_z f(u,t, z)$, $u\in [0,t]$, satisfies the following ODE:
\begin{eqnarray}
\frac{d}{du} X(u) &= & \frac{2X(u)}{f(u,t, z)^2} \label{eqn:prtlz},\\
X(0) &= &1. \nonumber
\end{eqnarray}
\item We claim that $z\mapsto f(u,t, z)$ is conformal in $\cC_\th1$.  Indeed it suffices to show that this map is injective. Suppose for some $u,t, z_1, z_2$ that
$$f(u,t, z_1)=f(u,t, z_2).$$
Let $Y_j(v)=f(u-v,t, z_j)$ for $v\in [0,u]$, $j=1,2$. Then $Y_j$'s satisfy the same ODE with the same initial value:
\begin{eqnarray*}
\partial_v (Y_j(v) + \lbda(t-u+v))&=&\frac{2}{Y_j(v)}, ~~~ v\in [0,u],\\
Y_j(0)&=&f(0,t, z_j).
\end{eqnarray*}
It implies that $Y_1(u)=Y_2(u)$. In particular, $z_1=z_2$. Hence $f(u,t, z)$ is conformal with respect to $z\in \cC_\th1$.

\item The map $f(u,t, z)$ satisfies a concatenation property; see the identity (\ref{eqn:con1}) in Section \ref{s:injectivity}.
\end{enumerate}
\end{remark}
%
%

\subsection{Existence of $f(u,t,0^+)$ and $\gamma^{(\lbda)}(t)$} \label{s:gamma}
In this subsection, we show that $f(u,t,0)=\lim_{ y\to 0^+} f(u,t , iy)$ exists. It follows from the equation (\ref{eqn:prtlz}) that 
$$\partial_z f(u,t ,z) = \exp \int^u_0 \frac{2}{f(v,t ,z)^2} \,dv.$$
Thus,
\begin{eqnarray*}
|\partial_z f(u,t,iy)| &= &\exp \int^u_0 \re\frac{2}{f(v,t,z)^2} dv \\
&= & \exp\int^u_0 \frac{2(x(v)^2-y(v)^2)}{(x(v)^2+y(v)^2)^2}\,dv~~ \mbox{ where } x(v)+iy(v)=f(v,t,z)\\
&\leq & 1 \mbox{ because } x(v)+iy(v)\in \cC_{\tau(0)}\subset \cC_1 \mbox{ and because of Remark \ref{rm:quant}(ii).}
\end{eqnarray*}
This shows that $f(u,t,iy)$ converges  to a limit, denoted by $f(u,t,0^+)$, uniformly in $u,t$ as $y\to 0^+$. Also
\begin{enumerate}[(i)]
\item $\im f(u,t,0^+) \geq  \nu(\th1)\sqrt{u}$.
\item $f(u,t,0^+)\in \cC_{\tau(0)}\subset \cC_1,$ when $u>0$.
\item $f(0,t,0^+)=0$.
\item Fix $u,t$ and fix $\theta<\th1$. Since $f(u,t,z)$ is conformal in $\cC_\th1$, by a property of conformal mappings, $f(u,t,z)$ converges uniformly to $f(u,t,0^+)$ as $z\to 0$ and $z\in \cC_\theta$.
\end{enumerate} 
Denote
$$\gamma^{(\lbda)}(t):=f(t,t,0^+) +\lbda(0)\mbox{ for each } t\in [0,1].$$

%
%
%
\subsection{Injectivity of $t\mapsto \gamma^{(\lbda)}(t)$}\label{s:injectivity}

Suppose there exist $0\leq t_1< t_2\leq 1$ such that $\gamma^{(\lbda)}(t_1)=\gamma^{(\lbda)}(s_2)$. That implies
$$f(t_1,t_1,0^+)= f(t_2, t_2,0^+).$$
The reader who is familiar with Loewner equation may recognize the idea from this part. If one can show that for every driving function, the generated curve never hits the real line, then by the concatenation property, the curve never hits itself.

Fix $\theta_1=\theta_1(\sigma)$ which is mentioned in Remark \ref{rm:quant}. Fix $z\in \cC_\th1$. The functions $u\mapsto f(u,t_1, f(t_2 - t_1,t_2, z))$ and $u\mapsto f(u+t_2-t_1, t_2,z)$ satisfy the same initial value ODE:
\begin{eqnarray*}
\frac{d}{du}[X(u)+\lbda(t_1-u) ]&=& \frac{-2}{X(u)} \mbox{ for } u\in [0,t_1],\\
X(0)&=&f(t_2-t_1, t_2,z).
\end{eqnarray*}
Technically, we have shown the above equation has solution if $f(t_2-t_1,t_2,z)\in \cC_\th1$. But it does not affect the main goal of this section. 
Hence 
\begin{equation}\label{eqn:con1}
f(u,t_1,f(t_2 - t_1, t_2,z)) = f(u + t_2-t_1,t_2,z)\mbox{ for } u\in [0,t_1].
\end{equation}
Let $\tilde\lbda(u)=\lbda(u+t_1)$. Let $\tilde f(u,t ,z)_{0\leq u\leq t\leq t_2 - t_1}$ be the solution of (\ref{eqn:ODEA}) in Lemma \ref{lm:existence} when $\lbda$ is replaced by $\tilde\lbda$. We note that $u\mapsto \tilde f(u, t_2-t_1,z)$ also satisfies
\begin{eqnarray*}
\frac{d}{du}[X(u)+ \lbda(t_2-u) ]&=& \frac{-2}{X(u)},\\
X(0)&=&z.
\end{eqnarray*}
Hence 
\begin{equation*}\label{eqn:con2}
\tilde f(u, t_2 - t_1,z) = f(u, t_2,z) \mbox{ for } u\in [0,t_2 - t_1].
\end{equation*}
Letting $u=t_2-t_1$ in the previous identity and  $u=t_1$ in (\ref{eqn:con1}), we find that 

$$f(t_2,t_2,z)=f(t_1,t_1, \tilde f(t_2-t_1,t_2-t_1,z)) \mbox{ for all } z\in \cC_\theta.$$
Let $z=iy\to i0^+$, using the assumption, 
$$f(t_1,t_1,0^+)=f(t_2,t_2,0^+)=f(t_1,t_1,w_0).$$
where $w_0 = \tilde{f}(t_2-t_1,t_2-t_1,0^+) \in \cC_{\tau(0)}\subset \cC_\th1$ since $t_2\neq t_1$. This leads to a contradiction since the conformal map $\cC_\th1\ni z\mapsto f(t_1,t_1,z)$ cannot send the interior point $w_0$ and the boundary point 0 to the same value.  Thus $t\mapsto \gamma^{(\lbda)}(t)$ is injective.

%
%
%
\subsection{When $\lbda$ is a real-valued function} \label{s:lbda_real}

Suppose that $\lbda$ is a real-valued function. We learn from the previous section that $f(\cdot, t,  z)$ is the unique solution of the ODE
\begin{eqnarray}
\partial_u (f(u,t,  z)+\lbda(t-u)) &= &\frac{-2}{f(u,t,  z)},~~~0\leq u\leq t \leq 1 \label{eqn:ODEf},\\
f(0, t, z) &= &z.\nonumber
\end{eqnarray}
We recall that this ODE has solution for all $z\in \H$ (\cite{L}). Let $\tilde\gamma$ be the Loewner curve generated by the real-valued driving function $\lbda$. For each $t\in [0,1]$, the backward Loewner equation
\begin{eqnarray}
\partial_u h_{u,t}(z) &=& \frac{-2}{h_{u,t}(z)-  \lbda(t-u)},\label{eqn:ODEh}\\
h_{0,t}(z) & = &z\in \H, \nonumber
\end{eqnarray}
has the  solution $h_{u,t}(z),u\in [0,t]$ for each $z\in \H$  and $\tilde\gamma(t)=\lim_{y\to 0^+} h_{t,t}(\lbda(t)+iy)$.\\
By the uniqueness of the solutions given initial values in (\ref{eqn:ODEf}) and (\ref{eqn:ODEh}), 
$$h_{u,t}(z)=f(u, t,z-\lbda(t))+ \lbda(t-u).$$
By setting $u=t$ and $z=iy+\lbda(t)$ and letting $y\to 0^+$, 
$$\tilde\gamma(t)=\gamma^{(\lbda)}(t).$$

%
%
%
%
%
\subsection{Proof of Theorem \ref{t:motion}}\label{s:motion}

We learn from the previous section that for every $\lbda\in \Lambda_\sigma$, and $z\in \cC_{\theta_1}$, the solution $f(\cdot, t,z)$ to (\ref{e:ODEf2}) exists uniquely.

Fix $\lambda\in \Lambda_\sigma$ and $\alpha\in \D$. Let $f(u,t,z,\alpha)$, $u\in [0,t]$, be the unique solution to the equation (\ref{e:ODEf2}) with $\lambda$ replaced by $\alpha\lambda\in \Lambda_\sigma$. By the dependency of ODE, for fixed $u,t,z$, the map $\alpha\mapsto f(u,t,z,\alpha)$ is analytic in $\alpha\in \D$.

By an abuse of notation, denote
$$\gamma^{(\alpha)}(t)=\lim_{y\to 0^+} f(t,t,iy, \alpha) + \alpha \lbda(0).$$
The existence of the limit is proved in Section \ref{s:gamma}. It is uniformly in $t$ and $\alpha$. By Section \ref{s:injectivity}--Section \ref{s:lbda_real}, the map $(\alpha,t)\in \D\times [0,1]\mapsto \gamma^{(\alpha)}(t)$ satisfies the conclusion of Theorem \ref{t:motion}.
\qed

\begin{remark}
As explained in the Introduction, Theorem \ref{thm:MSS} and Theorem \ref{t:motion} imply that $\gamma^{(\lambda)}$ is a quasi-arc for any $\lambda\in \Lambda_\sigma$. In particular, $\gamma^{(\lambda)}$ is a curve when $\lambda$ has small norm. 
One can also show the continuity by the standard methods, for example, as in \cite[Section 3]{JVL11}.
\end{remark}

%
%
%
%
%
\section{Loewner equation driven by complex-valued  functions} \label{s:hull}

Consider a continuous function $\lambda:[0,1]\to \C$. For each $z\in \C$, consider the initial value ODE
\begin{eqnarray*}
\partial_t g_t(z) &=& \frac{2}{g_t(z)-\lbda(t)},\\
g_0(z) & = & z.
\end{eqnarray*}
For each $z\in \C$, define the life-time of $z$ as
 $T_z=\sup\{t\in (0,1]: \mbox{ solution exists for all } s\in [0,t)\}$. Let $L_t=\{z\in \C: T_z\leq t\}$. As in the classical Loewner equation, by ODE theory (see \cite[Chapter4]{L}) one can show that
\begin{enumerate}[(i)]
\item For each $t$, the set $L_t$ contains $L_0=\{\lbda(0)\}$, is bounded. Also $\C\backslash L_t$ is open.
 
\item The map $g_t$ is analytic in $z\in \C\backslash L_t$ and $g_t(z)$ is jointly continuous in $(t,z)\in [0,T]\times C_T.$
 
\item If $T_z<1$ then $\lim_{t\to T_z^-} g_t(z)=\lbda(T_z).$

\item $\inf_{t\in [0,T]} |g_t(z)-\lbda(t)|>\delta_z>0$ for any $z\in \C\backslash L_T$.

\item The map $g_t$ is conformal at infinity and 
  $$g_t(z)=z+\frac{2t}{z}+O(\frac{1}{|z|^2}) \mbox{ when } z\to \infty.$$

\item  The map $g_t : \C\backslash L_t\to g_t(\C\backslash L_t)$ is conformal. Define
  $$R_t:=\C\backslash g_t(\C\backslash L_t)$$

  \noindent The left hull $L_t$ and right hull $R_t$ satisfy the following properties.

\item (Translation)
For any $a\in \C$,
\begin{eqnarray*}
L_{t, \lbda + a} & = & a + L_{t, \lbda},\\
R_{t, \lbda + a} & = & a + R_{t, \lbda}.
\end{eqnarray*}

\item (Scaling) For any $a>0$,
\begin{eqnarray*}
L_{t,a\lbda(\cdot/a^2)} & = & a L_{t/a^2,\lbda},\\
R_{t,a\lbda(\cdot/a^2)} & = & a R_{t/a^2,\lbda}.
\end{eqnarray*}

\item (Symmetry)
\begin{eqnarray*}
L_{t,-\lbda} & = & -L_{t,\lbda},\\
R_{t,-\lbda} & = & -R_{t,\lbda}.
\end{eqnarray*}

\item (Concatenation)
\begin{eqnarray*}
g_t(L_{t+s,\lbda}\backslash L_{t,\lbda})&\subset& L_{s,\lbda(t+\cdot)}\backslash R_{t,\lbda},\\
L_{t+s,\lbda} &=& L_{t,\lbda}\cup g_t^{-1} (L_{s,\lbda(t+\cdot)}\backslash R_{t,\lbda}),\\
L_{t+s,\lbda} &=& L_{t,\lbda}\cup g_t^{-1} (L_{s,\lbda(t+\cdot)}\backslash R_{t,\lbda}),
\end{eqnarray*}
where $\lbda(t+\cdot)$ is the map $r\mapsto \lbda(t+r)$.

\item (Duality)
\begin{eqnarray*}
L_{t,\lbda}&=&iR_{t,-i\lbda(t-\cdot)},\\
R_{t,\lbda}&=&iL_{t,-i\lbda(t-\cdot)}.
\end{eqnarray*}
\end{enumerate}

The duality property was a new property which was first observed by Rohde and Schramm \cite{RS13}.  It says roughly that what is true for the left hull is also true for the right hull and vice versa. We will make use of this fact later.

\subsection{Proof of Theorem \ref{t:L_t curve}}
Fix $\lbda\in \Lambda_\sigma$ with small $\sigma$. For each $\theta\geq 0$, denote
$$\cD_\theta= \cC_\theta\cup (-\cC_\theta),$$
which is a two-side cone based at 0. Fix $\theta_1=\th1(\sigma)>1$ which is defined in  Remark \ref{rm:quant}. Let $\theta_2$ be the number such that $\cD_{\theta_2}$ is the ``complement cone'' of $i\cD_\th1$, i.e.
$$\overline{\cD_{\theta_2}}=\overline{\C\backslash i\cD_\th1}.$$
To get the right picture, the reader may pretend $\theta_1$ is big and $\theta_2$ is small. Since   $-i\lbda(t)\in L_{0,-i\lbda(t-\cdot)}$, the duality property implies
$$\lbda(t)\in R_t.$$
The Lemma \ref{lm:existence} and the duality property implies
\begin{eqnarray*}
\lbda(t)+\cD_\th1&\subset &\C\backslash R_t \mbox{ and}\\
\lbda(0)+i\cD_\th1&\subset &\C\backslash L_t.
\end{eqnarray*}
Hence
\begin{eqnarray*}
L_t &\subset &\lbda(0)+\overline{\cD_{\theta_2}},\\
R_t & \subset & \lbda(t)+i\overline{\cD_{\theta_2}},\\
L_{s,\lbda(t+\cdot)} &\subset &  \lbda(t)+ \overline{\cD_{\theta_2}}\subset \lbda(t)+\overline{\cD_{\th1}} \mbox{ for } s\geq 0,
\end{eqnarray*}
where in the last claim, we use the fact $\theta_1>1$. Therefore, $L_{s,\lbda(t+\cdot)} \backslash  \{\lbda(t)\}$ lies in the conformal domain of $g^{-1}_t$:
\begin{equation}\label{e:Ls}
g^{-1}_t (L_{s,\lbda(t+\cdot)} \backslash \{\lbda(t)\}) \cap L_t = \emptyset.
\end{equation}

Now we observe that all arguments in Section \ref{s:proof} are still true (appropriately modified) if we study the equation (\ref{e:ODEf2}) with $z\in -C_\theta$. In particular, one can extend $\gamma$ to  $[-1,0]$ by letting
$$\gamma(t)=\lim_{y\to 0+} f(t,t,-iy) + \lambda(0) \mbox{ for } t\in [-1,0].$$
It follows from Section \ref{s:gamma} that $\gamma( t)\in \cC_1+\lbda(0)$ and $\gamma(-t)\in (-\cC_1)+\lbda(0)$ for each $t\geq 0$. 
Also  since $g_t$ is conformal from $\C\backslash L_t$ to $\C\backslash R_t$, $\lambda(t)\in R_t$, and
$$\gamma(t)=\lim_{y\to 0^+} g^{-1}_t(\lambda(t)+iy),$$
we obtain
$$\gamma(t)\in L_t\backslash \{\lbda(0)\}.$$
By the concatenation property, 
$$g_t(\gamma(t+s)) = \gamma^{(\lbda(t+\cdot))}(s)\in L_{s,\lbda(t+\cdot)}\backslash \{\lbda(t)\}.$$
Combine with (\ref{e:Ls}) and use the similarity,
$$\gamma(t+s),\gamma(-t-s)\notin L_t \mbox{ for } s>0.$$
It follows that $\gamma(t)$ and $\gamma(-t)$ are cut-points of $L_{t+s}$ for $s>0$. Since $\gamma$ is continuous, by a topology argument,
\begin{equation}\label{e:L_t union}
L_t=\gamma([-t,t]).
\end{equation}
Since $\gamma[0,t]$ and $\gamma[-t,0]$ are two quasi-arcs staying in two different cones $\cC_1+\lbda(0)$ and $-\cC_1+\lbda(0)$ respectively (except at the base $\gamma(0)=\lbda(0)$), the set $L_t$ is also a quasi-arc. \qed

%
%
%




\begin{remark} For each $\lbda:[0,1]\to \C$ with small H{\"o}lder norm, let $\gamma^{(\lbda)}:[-1,1]\to \C$ be the curve generated by $\lbda$ as in Theorem \ref{t:L_t curve}.

In the spirit of the papers \cite{Wong} and \cite{LT}, one may ask if their results hold for complex-valued functions. It turns out that they are  true without much modification. In particular, one can show the following results.
\end{remark}
\begin{prop}[Lemma 3.2 in \cite{Wong}]
Let $\sigma$ be small. And let $E_\sigma=\{\gamma^{(\lbda)}(1):\lbda\in \Lambda_\sigma, \lbda(0)=0\}$. Then there exists a constant $c>0$ such that
$$\diam (E_\sigma)\leq c\sigma.$$
\end{prop}
\noindent This proposition follows from the proof of Lemma \ref{lm:existence}; see Remark \ref{r:cone1}.
\begin{prop}[Theorem 3.4 in \cite{Wong}] Suppose $\lbda,\tilde\lbda:[0,1]\to \C$ both have small $1/2$-H{\"o}lder norm. Then 
$$||\gamma^{(\lbda)}-\gamma^{(\tilde\lbda)}||_{\infty,[-1,1]}\leq c||\lbda-\tilde\lbda||_{\infty,[0,1]},$$
where $c>0$ is a constant.
\end{prop}
As stated in \cite[pg. 1483]{Wong}, the above two propositions are the key tools of that paper.
\begin{prop}[Derivative's formula, Corollary 4.3 in \cite{Wong}] Suppose $\lbda:[0,1]\to \C$ is in $C^{1/2+\alpha}$ with $\alpha>0$, then $\gamma$ is differentiable on $(-1,1)\backslash \{0\}$ and
$$\gamma'(t)=\frac{i\sqrt{|t|}}{t}\exp\left\{\int^{|t|}_0 \left[\frac{1}{2u} + \frac{2}{(g_{|t|-u}(\gamma(t))-\lbda(|t|-u))^2}\right]du\right\},$$
for $t\neq 0$.

\end{prop}

\begin{prop}[Regularity result] If $\lbda\in C^{\alpha+1/2}([0,1])$ with $\alpha>1/2$, then $\gamma\in C^{\alpha+1/2}((-1,1)\backslash \{0\})$ quantitatively as in \cite{Wong} and \cite[Theorem 1.1]{LT}.
\end{prop}

%


\end{document}